\documentclass{article}
\usepackage{graphicx}
\usepackage{setspace}
\usepackage{amsfonts}
\usepackage{amssymb}
\usepackage{amsmath}
\usepackage{amsthm}
\usepackage{mathtools}
\usepackage{cases}
\usepackage{graphicx}
\usepackage[shortlabels]{enumitem}
\usepackage{tikz}
\usetikzlibrary{arrows.meta,calc,decorations.pathreplacing}
\newtheorem{thm}{Theorem}[section]

\newtheorem{lem}[thm]{Lemma}

\newtheorem*{thm2.4}{Theorem 2.4}

\theoremstyle{definition}
\newtheorem{defn}[thm]{Definition}

\newtheorem{rem}[thm]{Remark}
\newtheorem*{rem*}{Remark}
\newtheorem*{rems*}{Remarks}

\newtheorem*{ex*}{Example}

\title{Baker domains with arbitrarily slow controlled rate of escape}
\author{Benjamin Handsaker-Whetton}
\date{The Open University, Walton Hall, Milton Keynes, UK, MK7 6AA
benjamin.handsaker-whetton@open.ac.uk}

\begin{document}

\maketitle

\begin{abstract}
    Rempe and Rippon gave a general result which enables the construction of a transcendental entire function with a Baker domain, in which the iterates tend to infinity at a prescribed arbitrarily slow rate. Here we construct a transcendental entire function with a Baker domain, in which the iterates are controlled from both above and below in a prescribed arbitrarily slow but regular manner. The proof uses distortion theorems for conformal mappings due to Ahlfors and Warschawski, and an approximation theorem due to Rempe and Rippon. 
\end{abstract}

\section{Introduction}

\noindent Let $f : \mathbb{C} \rightarrow {\mathbb{C}}$ be an entire function and define the $n$th iterate of $f$ to be $f^n$, for $n = 0,1,2,...$. The \textit{Fatou set} $F(f)$ is defined to be the set of all points $z \in \mathbb{C}$ such that $\{ f^n : n \in \mathbb{N} \}$ forms a normal family in some neighbourhood of $z$. The \textit{Julia set} $J(f)$ is the complement of this set, $J(f) = {\mathbb{C}} \backslash F(f)$. We refer the reader to \cite{Bergweiler} for an introduction to the properties of these sets.

A \textit{Fatou component} of $f$ is a connected component of $F(f)$. If $U$ is a Fatou component of $f$, then $U_n$ is defined to be the Fatou component of $f$ such that $U_n \supset f^n(U)$. We call a Fatou component $U$ periodic if there exists $\ n \in \mathbb{N}$ such that $U_n = U$. For a transcendental entire function $f$, a periodic Fatou component $U$ is said to be a \textit{Baker domain} if, for some $k \in \mathbb{N}$, $f^{nk}(z) \rightarrow \infty$ as $n \rightarrow \infty$, for all $z \in U$. There are many examples of Baker domains; see for example \cite{BergweilerZhangBakerDomains} \cite{BaranskiKrzysztofFagellaJarque}, with the first example due to Fatou \cite[Example I]{FatousFunction}. For more background on Baker domains we refer the reader to \cite{PhilSurvey}. 

If you take any two points $z, z'$ in a Baker domain $U$, then there exists $C > 1$ such that $|f^n(z)| \leq C |f^n(z')|$ for large $n \in \mathbb{N}$ \cite[Lemma 7 (5)]{Bergweiler}. So points in a Baker domain do not tend to $\infty$ rapidly. This paper gives a method for constructing Baker domains in which the points tend to infinity at a prescribed arbitrarily slow rate, controlled both above and below.

\newpage

The \textit{escaping set} of a function $f$, denoted by $I(f)$, is defined as follows:
\begin{equation}
    \nonumber I(f) = \{ z : f^n(z) \rightarrow \infty \ \textnormal{as} \ n \rightarrow \infty \}.
\end{equation}

\noindent The escaping set for entire functions was first studied in detail by Eremenko \cite{Eremenko}, and many authors have studied points that escape at different rates, with the so-called fast escaping set playing a key role in complex dynamics \cite{GPFast}. Recently, Bergweiler and Rempe published a survey \cite{BergRempeSurvey} giving a summary of many of the main results about escaping sets and organising many open questions surrounding the topic. Rippon and Stallard studied points with slow rates of escape in \cite{SlowEscape} and introduced various subsets of the escaping set. In particular, for any positive sequence $a = (a_n)$ such that $a_n \rightarrow \infty$ as $n \rightarrow \infty$, they defined
\begin{equation}
    \nonumber I^a(f) = \{ z \in I(f) : |f^n(z)| = O(a_n) \ \textnormal{as} \ n \rightarrow \infty \},
\end{equation}

\noindent where the rate of escape is bounded above, and
\begin{equation}\label{Iaa}
    {I}_a^a(f) = \{ z \in I(f): \exists \; C = C(z) > 1 \; \textnormal{s.}\textnormal{t.} \; \frac{a_n}{C} \leq |f^n(z)| \leq Ca_n, \; \textnormal{for} \; n \geq N(z) \},
\end{equation}

\noindent where the rate of escape is controlled, both above and below.

Rempe and Rippon \cite[Theorem 1.1.1]{Lasse} showed that, given an arbitrary sequence $a = (a_n)$ of positive numbers such that $a_n \rightarrow \infty$ as $n \rightarrow \infty$, no matter how slowly, there exists a Baker domain $U$ such that for every $z \in U$, $z \in I^a(f)$. To do this they adapted an approximation theorem of Arakelian \cite{AApprox}, to allow them to control approximation by entire functions on simply connected sets.

It is natural to ask whether it is possible to construct Baker domains in which points belong to $I_a^a(f)$ where $a = (a_n)$ is a given positive increasing sequence, tending to $\infty$ arbitrarily slowly. We introduce new techniques which enable us to prove the following result.

\begin{thm}
    \vspace{1mm}
        \textit{Let $a = (a_n)$ be a positive strictly increasing sequence such that $a_n \rightarrow \infty$ as $n \rightarrow \infty$, $a_0 = 0$, and for some $n_0 \in \mathbb{Z}$ with $n_0 \geq 2$}
        
        \begin{equation}\label{TheoremCondition}
            1 \leq \frac{a_{n-1} - a_{n-2}}{a_n - a_{n-1}} < 2, \ \textnormal{for} \ n \geq n_0.
        \end{equation}
        
        \textit{Then there exists an entire function $f$ with a Baker domain $U$ such that, for any $z \in U$, there exists  $C = C(z) > 1$ and $N =N(z) \in \mathbb{N}$ such that \[\frac{a_n}{C} \leq |f^n(z)| \leq Ca_n, \ \textnormal{for all} \ n \geq N.\]}
\end{thm}

\noindent We note that condition \eqref{TheoremCondition} is satisfied by many positive increasing sequences; for example $a_n = \log n$ and  $a_n = n^{\varepsilon}$ with $0 < \varepsilon \leq 1$. All sequences that satisfy \eqref{TheoremCondition} must satisfy $a_n \leq Cn$, $n = 1,2,3,...$ for some $C > 0$.

\begin{rem}
    Any iterated logarithm $\log\cdot\cdot\cdot\log(n)$ gives an example satisfying condition \eqref{TheoremCondition}. We can also show that under certain conditions, we can apply the logarithm to a given sequence and show that \eqref{TheoremCondition} holds. A proof of this can be found in Theorem 4.1.
\end{rem}

\begin{rem}
    In Theorem 4.2 we show that is it possible to find sequences $(a_n)$ tending to $\infty$ arbitrarily slowly and satisfying \eqref{TheoremCondition}.
\end{rem}

To prove Theorem 1.1, we begin by constructing an unbounded simply connected domain and a model function defined on that domain with the desired properties. The construction of the model function is rather delicate; we outline the key steps in Section 2. In Section 3 we apply the approximation result given by Rempe and Rippon \cite[Lemma 4.2]{Lasse}, to find an entire function that has almost exactly the same behaviour as the model function.

\textbf{Acknowledgements.} I would like to give my warmest thanks to Gwyneth Stallard and Phil Rippon, my supervisors, for their deep insight, continued feedback and support on this paper. I would also like to extend these thanks to Vasiliki Evdoridou, my secondary supervisor, for helping me with a few of the ideas.

\section{The model function}

In this section we construct an unbounded simply connected domain and a model function $g$ with the properties required in Theorem 1.1, for a sequence $(a_n)$ satisfying the hypotheses in Theorem 1.1 with $n_0 = 2$. We will show at the end of Section 3 that the general case can be deduced from the case when $n_0 = 2$.

We construct a domain $S(a)$ as shown in Figure 1. Our aim is to construct a model function $g: S(a) \rightarrow S(a)$ with $g^n(0) \approx a_n$.

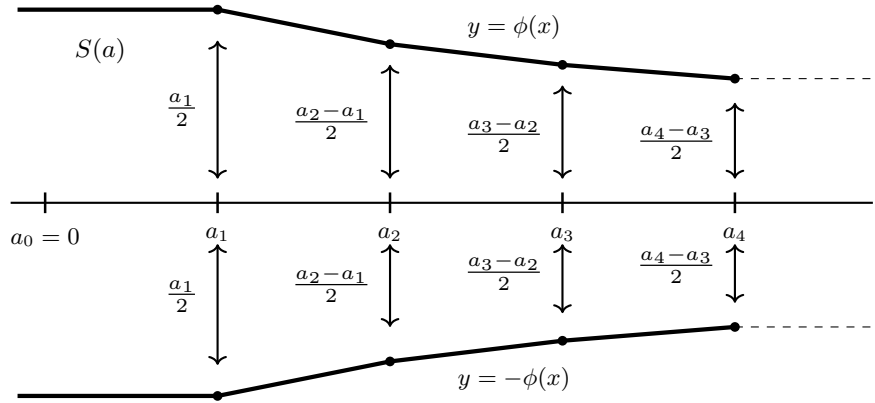
\begin{figure}[ht]
\begin{tikzpicture}[scale=0.9125, every node/.style={font=\small}]

\draw[thick] (-0.5,0) -- (12,0);

\foreach \x\lab in {0/{a_0=0}, 2.5/{a_1}, 5/{a_2}, 7.5/{a_3}, 10/{a_4} }
{
    \draw[line width=0.95pt] (\x,0.15)--(\x,-0.15);
    \node[below=6pt] at (\x,0) {$\lab$};
}

\draw[line width=1.6pt]
(-0.4,2.8) -- (2.5,2.8) -- (5,2.3) -- (7.5,2) -- (10,1.8);

\node at (6.8,2.55) {$y=\phi(x)$};

\fill (2.5,2.8) circle (2pt);
\fill (5,2.3) circle (2pt);
\fill (7.5,2) circle (2pt);
\fill (10, 1.8) circle (2pt);

\draw[line width=1.6pt]
(-0.4,-2.8) -- (2.5,-2.8) -- (5,-2.3) -- (7.5,-2) -- (10,-1.8);

\node at (6.8,-2.55) {$y=-\phi(x)$};

\fill (2.5,-2.8) circle (2pt);
\fill (5,-2.3) circle (2pt);
\fill (7.5,-2) circle (2pt);
\fill (10, -1.8) circle (2pt);

\node[scale=1.05] at (0.8,2.2) {$S(a)$};

\draw[<->,thick] (2.5,2.35) -- (2.5,0.35);
\node[scale = 1.25][left] at (2.4,1.35) {$\frac{a_1}{2}$};
\draw[<->,thick] (2.5,-0.6) -- (2.5,-2.35);
\node[scale = 1.25][left] at (2.4,-1.35) {$\frac{a_1}{2}$};

\draw[<->,thick] (5,2.0) -- (5,0.35);
\node[scale = 1.25][left] at (4.9,1.18) {$\frac{a_2-a_1}{2}$};
\draw[<->,thick] (5,-0.6) -- (5,-1.8);
\node[scale = 1.25][left] at (4.9,-1.18) {$\frac{a_2-a_1}{2}$};

\draw[<->,thick] (7.5,1.7) -- (7.5,0.35);
\node[scale = 1.25][left] at (7.4,1.0) {$\frac{a_3-a_2}{2}$};
\draw[<->,thick] (7.5,-0.6) -- (7.5,-1.6);
\node[scale = 1.25][left] at (7.4,-1.0) {$\frac{a_3-a_2}{2}$};

\draw[<->,thick] (10,1.45) -- (10,0.35);
\node[scale = 1.25][left] at (9.9,0.9) {$\frac{a_4-a_3}{2}$};
\draw[<->,thick] (10,-0.6) -- (10,-1.45);
\node[scale = 1.25][left] at (9.9,-0.9) {$\frac{a_4-a_3}{2}$};

\draw[dashed] (10,1.8) -- (12,1.8);
\draw[dashed] (10,-1.8) -- (12,-1.8);

\end{tikzpicture}

\caption{The domain $S(a)$}

\end{figure}

As shown in Figure 1, $S(a)$ is an unbounded domain lying between two piecewise linear curves with equations $y = \pm \phi(x)$.

In Section 3 we will use approximation to obtain an entire function $f$ which behaves in a similar way to $g$ inside $S(a)$, giving us a Baker domain with the desired properties.

Our choice of the domain $S(a)$ is influenced by existing results on the geometric nature of Baker domains; for example, \cite[Theorem 1.1]{BakerDomains2} shows that, for a Baker domain $U$ of a function $f$, the distance of $f^n(z)$ from the boundary of $U$ is closely related to the size of $|f^{n+1}(z) - f^n(z)|$.

In Section 2.1 we give the construction of $g$ and then in Sections 2.2 and 2.3 we carry out the calculations to show that $g$ has the required properties.

\subsection{Construction of the model function}

We start by taking a strip $S$ centred on the real axis and of height $\pi$ and mapping this into itself under the map $z \mapsto z + 1$.

The plan is to use a Riemann map to transfer this to a  self-map of $S(a)$ with the desired properties. The key here is that the hyperbolic distance between $a_n$ and $a_{n+1}$ in $S(a)$ is similar to the hyperbolic distance between $n$ and $n+1$ in~$S$.

We need to refine this rough idea to ensure that we end up with a function that maps $S(a)$ strictly inside itself, whilst still allowing enough space for the approximation that we need to do in Section 3.

It is easier to carry out this refinement in the right half-plane $\mathbb{H} = \{ x + iy : x > 0 \}$ than in the strip $S$ so we end up with the following set-up as illustrated in Figure 2.

\begin{enumerate}
    \item Let $\Phi$ be the Riemann map from $S(a)$ to $S$ that fixes 0 and maps the real axis to the real axis.

    \item Let $h: \mathbb{H} \rightarrow \mathbb{H}$ be defined to be $h(z) = \lambda z + 1$, $\lambda > 1$.

    \item Let the model map $g: S(a) \rightarrow S(a)$ be defined to be 
    \begin{equation}
        \nonumber g(z) = \Phi^{-1}(\textnormal{Log}(h(\exp(\Phi(z))))),
    \end{equation}

\noindent where $\textnormal{Log}$ denotes the principal branch of the logarithm.

\end{enumerate}

\newpage

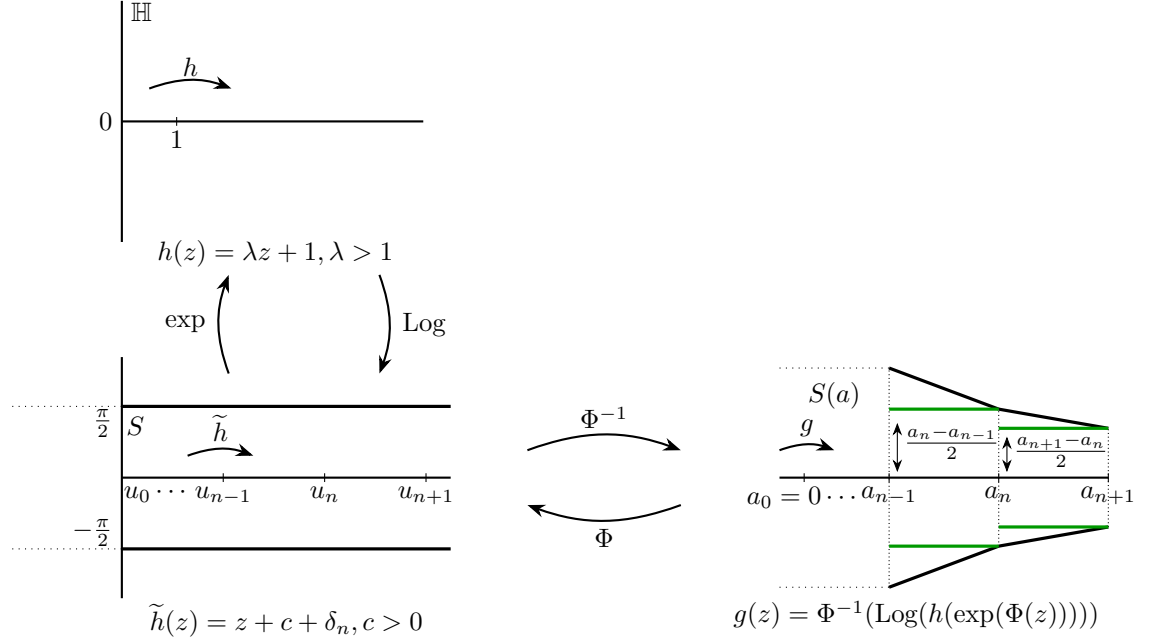
\begin{figure}[ht]

\begin{tikzpicture}[scale=0.725,>=Stealth]

\begin{scope}

\draw[thick] (0,-2.2)--(0,2.2);
\draw[thick] (0,0)--(5.5,0);

\node[left] at (0,0) {$0$};
\draw (1,-0.08)--(1,0.08);
\node[below] at (1,0) {$1$};
\node[right] at (0,2) {$\mathbb{H}$};

\node at (2.8,-2.4)
{$h(z) = \lambda z + 1, \lambda > 1$};

\draw[->,thick] (0.5,0.6) to[out=20,in=160] (2.0,0.6);
\node at (1.25,1) {$h$};

\end{scope}

\begin{scope}[xshift=4.2cm,yshift=-4cm]

\draw[->,thick] (-2.25,-0.6) to[out=110,in=250] (-2.25,1.2);
\node[left] at (-2.5,0.3) {exp};

\draw[->,thick] (0.5,1.2) to[out=-70,in=70] (0.5,-0.6);
\node[right] at (0.75,0.3) {Log};

\end{scope}

\begin{scope}[yshift=-6.5cm]

\draw[thick] (0,-2.2)--(0,2.2);
\draw[thick] (0,0)--(6,0);

\draw[very thick] (0,1.3)--(6,1.3);
\draw[very thick] (0,-1.3)--(6,-1.3);

\draw[dotted] (-2,1.3) -- (0,1.3);
\draw[dotted] (-2,-1.3) -- (0,-1.3);

\node[left] at (0,1) {$\frac{\pi}{2}$};
\node[left] at (0,-1) {$-\frac{\pi}{2}$};

\node at (3.0,-2.6)
{$\widetilde h(z) = z+c+\delta_n, c>0$};

\foreach \x\lab in {1.85/{u_{n-1}},3.7/{u_n},5.55/{u_{n+1}}}
{
\draw (\x,0.08)--(\x,-0.08);
\node[below] at (\x,0) {$\lab$};
}
\node at (0.65,-0.3) {$u_0\cdots$};

\node at (0.25,0.95) {$S$};

\draw[->,thick] (1.2,0.4) to[out=20,in=160] (2.4,0.4);
\node at (1.8,0.9) {$\widetilde h$};

\end{scope}

\begin{scope}[xshift=8.8cm,yshift=-6.5cm]

\draw[->,thick] (-1.4,0.5) to[out=20,in=160] (1.4,0.5);
\node at (0,1.1) {$\Phi^{-1}$};

\draw[->,thick] (1.4,-0.5) to[out=200,in=340] (-1.4,-0.5);
\node at (0,-1.1) {$\Phi$};

\end{scope}

\begin{scope}[xshift=12cm,yshift=-6.5cm]

\draw[thick] (0,0)--(6,0);

\foreach \x\lab in {0.45/{a_0=0\cdots},2/{a_{n-1}},4/{a_n},6/{a_{n+1}}}
{
\draw (\x,0.08)--(\x,-0.08);
\node[below] at (\x,0) {$\lab$};
}

\draw[densely dotted] (2,-2)--(2,2);
\draw[densely dotted] (4,-1.25)--(4,1.25);
\draw[densely dotted] (6,-0.9)--(6,0.9);
\draw[dotted] (0,2)--(2,2);
\draw[dotted] (0,-2)--(2,-2);

\draw[very thick] (2,2)--(4,1.25);
\draw[very thick] (2,-2)--(4,-1.25);

\draw[very thick] (4,1.25)--(6,0.9);
\draw[very thick] (4,-1.25)--(6,-0.9);

\draw[green!60!black,very thick] (2,1.25)--(4,1.25);
\draw[green!60!black,very thick] (2,-1.25)--(4,-1.25);

\draw[green!60!black,very thick] (4,0.9)--(6,0.9);
\draw[green!60!black,very thick] (4,-0.9)--(6,-0.9);

\draw[<->] (2.15,0.1)--(2.15,1);
\node[right] at (2.1,0.65) {$\frac{a_n-a_{n-1}}{2}$};

\draw[<->] (4.15,0.1)--(4.15,0.75);
\node[right] at (4.1,0.5) {$\frac{a_{n+1}-a_n}{2}$};

\node at (2.5,-2.5) {$g(z)=\Phi^{-1}(\textnormal{Log}(h(\exp(\Phi(z)))))$};

\node at (1,1.5) {$S(a)$};

\draw[->,thick] (0,0.5) to[out=20,in=160] (1,0.5);
\node at (0.5,0.9) {$g$};

\end{scope}

\end{tikzpicture}

\caption{Construction of the model function $g$; here $u_n = \Phi(a_n)$, for $n \geq 0$.}
\end{figure}

Our main goal in this section is to show that $g^n(0) \approx a_n$. In order to do this we must carry out some careful calculations to give bounds on the mapping $\Phi$.

\subsection{Properties of the conformal mapping}

In this subsection we establish properties of the conformal mapping $\Phi$, in particular obtaining estimates for the sequence $(\Phi(a_n))$. To be precise, we prove the following lemma.

First recall that, $y = \phi(x)$ is the equation of the upper boundary component of $S(a)$. We take $\phi_n(x) = \phi(x)$, $a_{n-1} \leq x \leq a_n$. Then $\phi_n(x)$ has gradient
\begin{equation}
    \nonumber \frac{\frac{a_n-a_{n-1}}{2} - \frac{a_{n-1}-a_{n-2}}{2}}{a_n - a_{n-1}} = \frac{1}{2} \left( 1 - \frac{a_{n-1}-a_{n-2}}{a_n - a_{n-1}} \right) = \frac{1}{2} \left( 1 - \frac{d_{n-1}}{d_n} \right) = \frac{-\varepsilon_n}{2},
\end{equation}

\noindent where $d_n = a_n - a_{n-1}$, $\varepsilon_n = \frac{d_{n-1}}{d_n} - 1$ and $0 \leq \varepsilon_n < 1$, because of \eqref{TheoremCondition}. Hence, 
\begin{equation}\label{LineEquations}
    \phi_n(x) = \frac{-\varepsilon_n}{2}(x-a_n) + \frac{d_n}{2}, \quad a_{n-1} \leq x \leq a_n.
\end{equation}

\noindent We note that the lower boundary of $S(a)$ has equation $y = -\phi(x)$.

\newpage

\begin{lem}
    \textit{Let $S$, $S(a)$ and $\Phi: S(a) \rightarrow S$ be as in Figure 2, and put $u_n = \Phi(a_n)$, $n \geq 1$. Then,}

    \begin{enumerate}[(a)]
        \item \begin{equation}\label{Lemma2.1(a)}
            \pi \left( n - \sum_{i=1}^n \frac{\varepsilon_i}{2} - 4 \right) \leq u_n \leq \pi \left( n + \frac{1}{12} \sum_{i=1}^n \varepsilon_i^2 + o(1) \right), \quad \textnormal{for } n \geq 1,
        \end{equation}
    \end{enumerate}

\textit{and there exists $N \geq 16$ such that}

\begin{enumerate}[(b)]
    \item \begin{equation}\label{Lemma2.1(b)}
        \frac{1}{4} \pi n \leq u_n \leq \frac{13}{6} \pi n, \quad \textnormal{for } n \geq N.
    \end{equation}
\end{enumerate}

\end{lem}

We will prove part (a) by using two distortion theorems, one due to Ahlfors \cite[p.97]{Nevanlinna} and the other due to Warschawski \cite[Theorem IV]{Warschawski}.

\begin{thm}[Ahlfors' distortion theorem]\label{AhlforsThm}

    \textit{Let $\Omega$ be a simply connected domain in $\mathbb{C}$ such that}

\begin{equation}
    \nonumber \theta(x_0) = \{ z \in \Omega : \Re(z) = x_0 \}, \ x_0 \in \mathbb{R},
\end{equation}

\noindent \textit{is either empty or a bounded line segment of length $\theta(x_0)$. Let $\varphi$ be a conformal mapping from $\Omega$ to the strip $S = \{ w = u + iv: |v| < \frac{\pi}{2} \}$ such that $z = x + iy \in \Omega \mapsto w = u + iv \in S$ and $\lim_{x \rightarrow +\infty} \varphi(x + iy) = +\infty$. If $z_1 = x_1 + iy_1$ and $z_2 = x_2 + i y_2$ map to $w_1 = u_1 + iv_1$ and $w_2 = u_2 + iv_2$ respectively, then,}

\begin{equation}\label{Ahlfors}
    \pi \int_{x_1}^{x_2} \frac{dx}{\theta(x)} \leq u_2 - u_1 + 4 \pi.
\end{equation}

\end{thm}

\noindent Theorem 2.2 will allow us to estimate $u_n = \Phi(a_n)$ from below. In the other direction we use the following.

\begin{thm}[Warschawski's distortion theorem]\label{WarschawskiThm}
    \vspace{1mm}
\textit{Let $\phi_1(x)$, $\phi_2(x) \in \mathbb{C}^0(-\infty,\infty)$ with uniformly bounded derivatives almost everywhere and with $\phi_1(x) < \phi_2(x)$, for $-\infty < x < \infty$. Let}

\begin{equation}
    \nonumber \Omega = \{z \in \mathbb{C} : -\infty < x < \infty, \phi_1(x) < y < \phi_2(x) \},
\end{equation}
\begin{equation}
    \nonumber \theta(x) = \phi_1(x) - \phi_2(x) \quad \textnormal{and} \quad \psi(x) = \frac{1}{2}( \phi_1(x) + \phi_2(x)).
\end{equation}

\noindent \textit{Let $\varphi$ be a conformal mapping from $\Omega$ to the strip $S$ as in Theorem \ref{AhlforsThm}.}

\textit{We say that $\Omega$ is an L-strip with the boundary inclination $\gamma$ at $x = +\infty$, $|\gamma| < \pi/2$, if, for $x_2 > x_1$,}
\begin{equation}
    \nonumber \frac{\phi_1(x_2) - \phi_1(x_1)}{x_2 - x_1}, \quad \frac{\phi_2(x_2) - \phi_2(x_1)}{x_2 - x_1}
\end{equation}

\noindent \textit{approach the same limit, $\tan\gamma$, as $x_1$ and $x_2 \rightarrow +\infty$ simultaneously.}

\textit{If $\Omega$ is an L-strip with boundary inclination $\gamma = 0$ at $x = +\infty$, then}

\begin{equation}\label{Warschawski}
    u_2 - u_1 \leq \pi \int_{x_1}^{x_2} \frac{1+\psi'^2(x)}{\theta(x)}dx + \frac{\pi}{12} \int_{x_1}^{x_2} \frac{\theta'(x)^2}{\theta(x)} dx + o(1),
\end{equation}

\noindent \textit{as $u_1, u_2 \rightarrow +\infty$, uniformly with respect to $v_1$ and $v_2$.}

\end{thm}

\noindent We note that in our case the domain $S(a)$ is symmetric with respect to the real axis so we have $\psi(x) \equiv 0$. We also note that \cite[Theorem. IV]{Warschawski} requires the derivatives of $\phi_1$ and $\phi_2$ to exist everywhere, but it is clear from the discussion in \cite[p.290]{Warschawski} that the proof is valid when the functions $\phi_1$ and $\phi_2$ are differentiable almost everywhere with uniformly bounded derivative.

\begin{proof}[Proof of Lemma 2.1]

To apply Theorems 2.2 and 2.3 we first need to estimate,

\begin{equation}
    \int_{a_0}^{a_n} \frac{dx}{\theta (x)} = \sum_{i=0}^{n-1} \int _{a_i}^{a_{i+1}} \frac{dx}{\theta_i(x)},
\end{equation}

\noindent where $\theta(x) = 2 \phi(x)$ and $\theta_n(x) = 2\phi_n(x)$ for $a_{n-1} \leq x \leq a_n$. We define $\tilde{\phi}_n(x) = (a_n-a_{n-1})/2$, $a_{n-1} \leq x \leq a_n$, and we take $\tilde{\phi}(x) = \tilde{\phi}_n(x)$, $a_{n-1} \leq x \leq a_n$, $n \geq 1$. If $\varepsilon_n = 0$, then, by \eqref{LineEquations},

\begin{equation}
    \nonumber \int_{a_{n-1}}^{a_n} \frac{dx}{\theta_n(x)} = 1.
\end{equation}

\noindent If $0 < \varepsilon_n < 1$, then we first note that
\begin{equation}\label{Boxes}
    \int_{a_{n-1}}^{a_n} \frac{dx}{\tilde{\theta}(x)} = \int_{a_{n-1}}^{a_n} \frac{dx}{a_n - a_{n-1}} = 1, \quad \textnormal{where} \quad \tilde{\theta}(x) = 2\tilde{\phi}(x),
\end{equation}

\noindent and estimate the following difference, where $\tilde{\theta}_n(x) = 2\tilde{\phi}_n(x)$, $a_{n-1} \leq x \leq a_n$:
\begin{equation}\label{0<epsilon<1}
    \begin{split}
        \int_{a_{n-1}}^{a_n} \left(\frac{1}{\tilde{\theta}_n(x)} - \frac{1}{\theta_n(x)} \right) dx & = 1 - \int_{a_{n-1}}^{a_n} \frac{1}{2( \frac{-\varepsilon_n}{2} (x-a_n) + \frac{d_n}{2})}  dx \\
        & = 1 - \left[ \frac{-1}{\varepsilon_n} \log(-\varepsilon_n(x-a_n) + d_n )\right]_{a_{n-1}}^{a_n} \\
        & = 1 + \frac{1}{\varepsilon_n} \left[ \log d_n - \log(-\varepsilon_n(a_{n-1} -a_n ) + d_n) \right] \\
        & = 1 + \frac{1}{\varepsilon_n} \left[ \log d_n - \log((-\varepsilon_n)(-d_n) + d_n) \right] \\
        & = 1 - \frac{1}{\varepsilon_n} \log (1+\varepsilon_n) \\
        & = 1 - \frac{1}{\varepsilon_n} (\varepsilon_n - \frac{1}{2} \varepsilon_n^2 +\frac{1}{3} \varepsilon_n^3 - ...) \\
        & = \frac{1}{2} \varepsilon_n - \frac{1}{3} \varepsilon_n^2 + ... \\
        & < \frac{1}{2} \varepsilon_n, \quad 0 < \varepsilon_n < 1.
    \end{split}
\end{equation}

\newpage

\noindent We have

\begin{equation}
    \nonumber \int_{a_0}^{a_n} \left( \frac{1}{\tilde{\theta}(x)} - \frac{1}{\theta(x)} \right) dx = \sum_{i = 0}^{n-1} \int_{a_i}^{a_{i+1}} \left( \frac{1}{\tilde{\theta}_i(x)} - \frac{1}{\theta_i(x)} \right) dx,
\end{equation}

\noindent so, by \eqref{Boxes} and \eqref{0<epsilon<1},

\begin{equation}\label{n - epsilon}
    n - \int_{a_0}^{a_n} \frac{dx}{\theta(x)} \leq \sum_{i = 1}^{n} \frac{\varepsilon_i}{2}, \quad \textnormal{that is,} \quad \int_{a_0}^{a_n} \frac{dx}{\theta(x)} \geq n - \sum_{i=1}^n \frac{\varepsilon_i}{2}.
\end{equation}

\noindent We can then feed this into Theorem 2.2 to give the required lower bound for $u_n - u_0 = \Phi(a_n) - \Phi(a_0)$. We recall here that $a_0 = 0$ and so $\Phi(a_0) = 0$. 

Now, we find an upper bound for $\int_{a_0}^{a_n} \frac{\theta'(x)^2}{\theta(x)} dx$ in order to apply Theorem~2.3 to give us an upper bound for $u_n - u_0$. We have

\begin{equation}\label{epsilonsquared}
    \int_{a_{n-1}}^{a_n} \frac{\theta_n'(x)^2}{\theta_n(x)} dx = \int_{a_{n-1}}^{a_n} \frac{\varepsilon_n^2}{\theta_n(x)} dx \leq \varepsilon_n^2,
\end{equation}

\noindent since

\begin{equation}
    \nonumber \int_{a_{n-1}}^{a_n} \frac{dx}{\theta_n(x)} \leq \int_{a_{n-1}}^{a_n} \frac{dx}{\tilde{\theta}_n(x)} \leq 1.
\end{equation}

\noindent This proves part (a).

To deduce part (b) from part (a), we use the fact that $0 \leq \varepsilon_n < 1$. First
\begin{equation}\label{Upper}
    \begin{split}
        u_n  & \leq \pi \left( n + \frac{1}{12} \sum_{i=1}^n\varepsilon_i^2 + o(1) \right) \\
        & \leq \pi \left( n + \frac{1}{12}n + o(1 ) \right) \\
        & \leq \pi \left( \frac{13n}{12} + 1 \right), \ \textnormal{for } n \geq N \\
        & \leq \frac{13 \pi n}{6}, \ \textnormal{for} \ n \geq N.
    \end{split}
\end{equation}

\noindent Then,
\begin{equation}\label{Lower}
    \begin{split}
        u_n & \geq \pi \left( n - \sum_{i=1}^n \frac{\varepsilon_i}{2} - 4 \right) \\
        & \geq \pi \left(n - \frac{n}{2} - 4 \right) \\
        & = \pi \left( \frac{n}{2} - 4 \right) \\
        & \geq \frac{\pi n}{4}, \ \textnormal{for} \ n \geq 16.
    \end{split}
\end{equation}

\noindent This completes the proof of Lemma 2.1.
\end{proof}

\subsection{Dynamics of the model function}

Recall that the model function $g$ is a self map of $S(a)$ defined in terms of a self map of $\mathbb{H} = \{ x + iy : x > 0 \}$ given by
\begin{equation}
    \nonumber h(z) = \lambda z + 1, \quad \textnormal{where} \ \lambda > 1.
\end{equation}

This map $h$ is conjugated to the self map $\tilde{h}$ of $S$ given by
\begin{equation}
    \nonumber \tilde{h}(z) = \textnormal{Log}(h(\exp(z))), \ z \in S,
\end{equation}

\noindent and then to the self map $g$ of $S(a)$ given by $g = \Phi^{-1} ( \tilde{h} ( \Phi))$.

We show that, for sufficiently large $n$, $g^n(0)$ is comparable to $a_n$.

\begin{lem}
    \textit{With $(a_n)$ and $g$ defined as above, we have}
    \begin{equation}
        \nonumber \frac{1}{5} a_n \leq g^n(0) \leq 5 a_n, \quad \textnormal{for} \ n \geq N \geq 16.
    \end{equation}
\end{lem}

\begin{proof}

    First we estimate the behaviour of $\tilde{h}^n(0)$. For $n \geq 0$,
\begin{equation}
\nonumber 
    \begin{split}
        \tilde{h}^n(0) = \log(h^n(e^0)) = \log(h^n(1)) & = \log \left( \sum_{k=0}^{n} \lambda^k \right),\\
        & = \log(\lambda^n) + \log \left( 1 + \sum_{k = 1}^{n} \frac{1}{\lambda^k} \right), \\
    \end{split}
\end{equation}

\noindent and if we choose $\lambda = e^\pi$, then $0< \sum_{k=1}^n \frac{1}{\lambda^k} < 1$ and $0 < \log \left( 1 + \sum_{k = 1}^{n} \frac{1}{\lambda^k} \right) < 1$. So we have that there exists $\delta_n$ such that $0 < \delta_n < 1$ and

\begin{equation}\label{Mapping0inModel}
    \tilde{h}^n(0) = \log(\lambda^n) + \log \left( 1 + \sum_{k = 1}^{n} \frac{1}{\lambda^k} \right) = n \log(e^\pi) + \delta_n = \pi n + \delta_n.
\end{equation}

\noindent Now we would like to find a relationship between $a_n$ and $a_{Cn}$ with $C \in \mathbb{N}$. By the definition of the $d_n$ we have the following,

\begin{equation}
    \nonumber a_n = \sum_{i=1}^n d_i, \quad n \geq 1.
\end{equation}

\noindent So for any $C \in \mathbb{N}$ we have that,

\begin{equation}
    \nonumber a_{Cn} = \sum_{i=1}^{Cn} d_i.
\end{equation}

\noindent Since we have that $d_n \leq d_{n-1}$, we also have the following:

\begin{equation}
    \nonumber \sum_{i = m}^{n+m-1} d_i \leq \sum_{i=1}^n d_i = a_n, \quad \textnormal{for all } m,n \geq 1.
\end{equation}

\noindent So we conclude that

\begin{equation}\label{MultipleSubscripts}
    a_{Cn} = \sum_{i=1}^{Cn} d_i \leq Ca_n.
\end{equation}

From Lemma 2.1 (b), we have that $u_n \in [ \pi n/4 , 13 \pi n/6] = I_n$, say, for $n \geq N$, where $N \geq 16$. Since we have that $\tilde{h}^n(0) = \pi n + \delta_n$, $0 < \delta_n < 1$ by \eqref{Mapping0inModel}, we want to find which intervals $I_m$ contain $\tilde{h}^n(0)$, expressed in terms of $m$.

To do this we observe that $\tilde{h}^n(0) \not\in I_m$ if and only if

\begin{equation}
    \nonumber
        \textnormal{either } \quad \pi n + \delta_n < \frac{\pi m}{4}, \quad \textnormal{or} \quad \pi n + \delta_n > \frac{13 \pi m}{6}.
\end{equation}

\noindent That is,

\begin{equation}
\nonumber 
        \textnormal{either } \quad m > \frac{\pi n + \delta_n}{ \frac{\pi}{4}}, \quad \textnormal{or} \quad m < \frac{\pi n + \delta_n}{ \frac{13 \pi }{6} }.
\end{equation}

\noindent Therefore, for $n \geq N$, $g^n(0)$ must lie in $[a_{p(n)}, a_{q(n)}]$, where

\begin{equation}
    \nonumber p(n) = \left\lfloor \frac{\pi n + \delta_n}{\frac{13 \pi }{6}} \right\rfloor, \quad q(n) = \left\lfloor \frac{\pi n + \delta_n}{\frac{\pi}{4}} \right\rfloor + 1.
\end{equation}

\noindent Since, by \eqref{Mapping0inModel}, we have that $0 < \delta_n < 1$,

\begin{equation}
    \nonumber
        p(n) = \left\lfloor \frac{\pi n + \delta_n}{\frac{13 \pi }{6}} \right\rfloor = \left\lfloor \frac{6n + 6\delta_n/\pi}{13} \right\rfloor \geq \left\lfloor \frac{6 n}{13} \right\rfloor.
\end{equation}

\noindent So, by \eqref{MultipleSubscripts}, we have that

\begin{equation}
    a_{p(n)} \geq a_{ \left\lfloor 6n/13 \right\rfloor} \geq \frac{1}{5}a_{5 \left\lfloor 6n/13 \right\rfloor} \geq \frac{1}{5}a_n, \ \textnormal{for} \ n \geq 3.
\end{equation}

\noindent We also have that

\begin{equation}
    \nonumber
    \begin{split}
        q(n) & = \left\lfloor \frac{\pi n + \delta_n}{\frac{\pi}{4}} \right\rfloor + 1 \leq \left\lfloor \frac{\pi n + 1}{\frac{\pi}{4}} \right\rfloor + 1 \\
        & = \left\lfloor 4n + \frac{4}{\pi} \right\rfloor + 1 = 4n + 2 \leq 5n, \ \textnormal{for} \ n \geq 2,
    \end{split}
\end{equation}

\noindent and so,

\begin{equation}
    \nonumber a_{q(n)} \leq a_{5n} \leq 5a_n, \ \textnormal{for } n \geq 2.
\end{equation}

\noindent Hence, we can conclude that $\frac{a_n}{5} \leq |g^n(0)| \leq 5 a_n$, for $n \geq N$ by \eqref{Lemma2.1(b)}, completing the proof of Lemma 2.4.
\end{proof}

\section{Approximating the model function}

\noindent The final key idea is to apply an approximation theorem to our model function~$g$ to obtain an entire function $f$ satisfying Theorem 1.1. We will do this in this section.

Here we state a special case of a result of Rempe and Rippon \cite[Lemma 4.2]{Lasse} in a form that is applicable to our problem.

First, recall that a Weierstrass set is defined as follows \cite[p.142]{AApprox}.

\begin{defn}\label{Def}
    Supposed a set $F$ is closed in $\mathbb{C}$. Then $F$ is called a \textit{Weierstrass set} in $\mathbb{C}$ if any continuous $f : F \rightarrow \mathbb{C}$ that is holomorphic on the interior of $F$ can be uniformly approximated by a holomorphic function $g: \mathbb{C} \rightarrow \mathbb{C}$, up to an arbitrarily smaller error $\varepsilon$. 
\end{defn}

\noindent We have the following sufficient condition for a set to be a Weierstrass set \cite[p.142]{AApprox}.

\begin{thm}\label{WeierstrassSet}
    \vspace{1mm}
    \textit{A closed set $F$ in $\mathbb{C}$ is a Weierstrass set if and only if}

    \begin{equation}
        \nonumber \hat{\mathbb{C}} \backslash F \ \textnormal{\textit{is connected}} \quad \textnormal{and} \quad \hat{\mathbb{C}} \backslash F \ \textnormal{\textit{is locally connected at }} \infty,
    \end{equation}    

    \noindent \textit{where $\hat{\mathbb{C}} = \mathbb{C} \cup \{\infty\}$ is the one-point compactification of $\mathbb{C}$.}
\end{thm}

\noindent Now we state the approximation theorem that we use to construct our example.

\begin{lem}\label{Approximation}
    \vspace{1mm}
    \textit{Let $V \subset \mathbb{C}$ be a simply connected domain such that $\overline{V}$ is a Weierstrass set. Let $\varphi : \mathbb{H} \rightarrow V$ be a Riemann map, and set $V' := \varphi( \{ \zeta \in \mathbb{H} : \Re(\zeta) \geq 1 \})$. Now let $g : \bar{V} \rightarrow \mathbb{C}$ be continuous on $\overline{V}$ and holomorphic on $V$. Suppose furthermore that $g(\overline{V}) \subset V'$. Then for every $\varepsilon > 0$, there exists an entire function $f$ such that,}

\begin{equation}\label{ApproximationResultLasse}
    \textnormal{for all } z\in V, \quad f(z) \in V \ \textnormal{and} \ |\varphi^{-1}(f(z)) - \varphi^{-1}(g(z)) | \leq \varepsilon.
\end{equation}
    
\end{lem}

Now, we state a theorem of Carathéodory \cite[Thm 3.1]{Garnett&Marshall} that shows the conformal mappings we use have a continuous extension to the boundary, which we need in order to apply Lemma 3.3.

\begin{thm}[Carathéodory]\label{Carathéodory}
    Let $\varphi$ be a conformal mapping from the unit disc $\mathbb{D}$ onto a Jordan domain $\Omega$. Then $\varphi$ has a continuous extension to $\overline{\mathbb{D}}$, and the extension is a one-to-one map from $\overline{\mathbb{D}}$ onto $\overline{\Omega}$.
\end{thm}

We would like to apply Lemma 3.3  to the model function 

\begin{equation}
    \nonumber g(z) = \Phi^{-1} ( \log ( h ( \exp ( \Phi(z))))),
\end{equation} 

\noindent which will give us an entire function $f: \mathbb{C} \rightarrow \mathbb{C}$ that shares similar dynamics to $g$ on $V = S(a)$, where $S(a)$ is as in Figure 1. From the model we can see that $S(a) \subset \mathbb{C}$ is simply connected, and $\overline{S(a)}$ is a Weierstrass set since $\hat{\mathbb{C}} \backslash \overline{S(a)}$ is connected and locally connected at $\infty$. As in the model we also have a Riemann map $\varphi: \mathbb{H} \rightarrow S(a)$, where $\varphi(z) = \Phi^{-1} ( \textnormal{Log}z)$. Since both $\mathbb{H}$ and $S(a)$ are Jordan domains, by Carathéodory's Theorem \eqref{Carathéodory}, the Riemann map $\varphi(z) = \Phi^{-1}(\textnormal{Log}z)$ has a continuous extension to the respective boundaries $\overline{\mathbb{H}}$ and $\overline{S(a)}$. The function $g: \overline{S(a)} \rightarrow \mathbb{C}$ is continuous on $\overline{S(a)}$ and holomorphic on $S(a)$ and we also have that $g(\overline{S(a)}) \subset S(a)'$, where $S(a)' = \varphi(\{\zeta \in \mathbb{H} : \Re(\zeta) \geq 1 \} )$. We note that this is true since $g: S(a) \rightarrow S(a)$ is conjugate to $h : \mathbb{H} \rightarrow \mathbb{H}$ under $\Phi^{-1} ( \textnormal{Log})$, $h(\mathbb{H}) \subset \{\zeta \in \mathbb{H}: \Re(\zeta) > 1 \}$ and Carathéodory's Theorem \eqref{Carathéodory} continuously extends this to the boundary. Hence, we can apply Lemma 3.3 with $V = S(a)$, $\overline{V} = \overline{S(a)}$ and $V' = S(a)'$. 

We now need to check to what extent applying this approximation will change the dynamics in $\mathbb{C}$ of the resulting function $f$.

From Lemma 3.3, we have that, for $z \in S(a)$, $f(z) \in S(a)$ and $|\varphi^{-1}(f(z)) - \varphi^{-1}(g(z))| \leq \varepsilon$, where $0 < \varepsilon < 1$. We let

\begin{equation}
    \nonumber
    \begin{split}
        G(z) & = \varphi^{-1} ( g ( \varphi(z))) = \lambda z + 1 \\
        F(z) & = \varphi^{-1} ( f ( \varphi(z))),
    \end{split}
\end{equation}

\noindent both of which map $\mathbb{H}$ into itself. We also note that 
\begin{equation}
    \nonumber G^n(1) = \varphi^{-1} ( g^n(0)) = \lambda^n + \cdot \cdot \cdot + \lambda + 1, 
\end{equation}

\noindent since $\varphi(1) = \Phi ( \textnormal{Log} \ 1) = \Phi(0) = 0$. Then from \eqref{ApproximationResultLasse} we have that,

\begin{equation}
    \nonumber |F(1)  - G(1)| = |\varphi^{-1}(f(0)) - \varphi^{-1}(g(0))| \leq \varepsilon.
\end{equation}

\noindent By the triangle inequality, and \eqref{ApproximationResultLasse},

\begin{equation}
    \nonumber
    \begin{split}
        |F^2(1) - G^2(1)| & \leq | F(F(1)) - G(F(1))| + |G(F(1)) - G(G(1))| \\
        & \leq \varepsilon + |G(F(1)) - G(G(1))| \\
        & = \varepsilon + | \lambda F(1) + 1 - (\lambda G(1) + 1)| \\
        & = \varepsilon + \lambda|F(1) - G(1)| \\
        & \leq \varepsilon + \lambda \varepsilon = \varepsilon(1 + \lambda) = \varepsilon G(1).
    \end{split}
\end{equation}

\noindent Applying this argument repeatedly, we obtain, for $n \geq 1$,

\begin{equation}
    \nonumber |F^n(1) - G^n(1)| \leq \varepsilon G^{n-1}(1),
\end{equation}

\noindent so
\begin{equation}\label{Symmetry}
        \left|\frac{F^n(1)}{G^n(1)} - 1 \right| \leq \varepsilon \frac{G^{n-1}(1)}{G^n(1)}.
\end{equation}

\noindent We now note, since the set $V$ and the function $g$ are symmetric with respect to the real axis, that we can take the approximating entire function $f$ to be symmetric about the real axis \cite[page 34]{Symmetric}. Since the values of $G^n(1)$ come from $\varphi(g^n(0))$ which are the forward iterates of $0$ in $S(a)$ and fall on the real axis, we have that $G^n(1)$ are on the real axis in $\mathbb{H}$. Hence, we also have that $F^n(1)$ are on the real axis in $\mathbb{H}$ and so from \eqref{Symmetry} we have

\begin{equation}
    \nonumber 1 - \varepsilon \frac{G^{n-1}(1)}{G^n(1)} \leq \frac{F^n(1)}{G^n(1)} \leq 1 + \varepsilon \frac{G^{n-1}(1)}{G^n(1)}.
\end{equation}

\noindent Since

\begin{equation}
    \nonumber \frac{G^{n-1}(1)}{G^n(1)} = \frac{G^{n-1}(1)}{\lambda G^{n-1}(1) + 1} \leq \frac{1}{\lambda},
\end{equation}

\noindent we obtain

\begin{equation}
    \nonumber 1 - \frac{\varepsilon}{\lambda} \leq \frac{F^n(1)}{G^n(1)} \leq 1 + \frac{\varepsilon}{\lambda}.
\end{equation}

\noindent Now to see how this inequality affects the dynamics of $f$ in $S$, we take the logarithm,

\begin{equation}
    \nonumber -\frac{2\varepsilon}{\lambda} \leq \log( 1 - \frac{\varepsilon}{\lambda}) \leq \log\left( \frac{F^n(1)}{G^n(1)} \right) \leq \log(1 + \frac{\varepsilon}{\lambda}) \leq \frac{\varepsilon}{\lambda}, \textnormal{ for } \varepsilon \textnormal{ small}.
\end{equation}

\noindent So we have,
\begin{equation}
    \nonumber \left|\log \left(\frac{F^n(1)}{G^n(1)} \right) \right| \leq \frac{2\varepsilon}{\lambda};
\end{equation}

\noindent that is,

\begin{equation}
        |\log (F^n(1)) - \log(G^n(1)) | \leq \frac{2\varepsilon}{\lambda}.
\end{equation}

\noindent We note that $\log(F^n(0))$ and $\log(G^n(0))$ have the orbit of their iterates of $\tilde{h}$ on $\mathbb{R} \subset S$. We also note that $\frac{2\varepsilon}{\lambda}$ is a fixed constant for all $n$ and can be as small as we like. So on $S$ we have that $\tilde{h}^n(0) = \pi n + \delta_n + \varepsilon_n$ where $|\varepsilon_n| \leq \frac{2 \varepsilon}{\lambda}$. From \eqref{Mapping0inModel}, we have that $\delta_n = \log\left(1 + \sum_{k=1}^n 1/\lambda^k \right) = \log\left(1 + \sum_{k=1}^n 1/e^{\pi k} \right)$, since $\lambda = e^\pi$. We have that $\frac{1}{e^\pi} < \sum_{k=1}^n \frac{1}{(e^\pi)^k} < 1$ and so,
\begin{equation}
    \nonumber \log\left( 1 + \frac{1}{e^\pi} \right) \leq \log \left(1 +\sum_{k=1}^n \frac{1}{e^{\pi k}} \right) < \log(1 + 1) < 0.7.
\end{equation}

\noindent Hence, if we take $\varepsilon$ small enough that $\frac{2\varepsilon}{\lambda} < \log\left( 1 + \frac{1}{e^\pi} \right)$, then we have that for all $n \in \mathbb{N}$, $0 < \delta_n +\varepsilon_n < 1$, say $\delta_n'$. We note that if we go through the arguments in Section 3.2 and replace $\delta_n$ with $\delta_n'$ then we obtain the same conclusion. Hence, we have that, for some $N \geq 16$,

\begin{equation}
    \nonumber \frac{a_n}{5} \leq |f^n(0)| \leq 5 a_n, \ \textnormal{for}  \ n \geq N.
\end{equation}

By Lemma 3.3, we have that $f(S(a)) \subset S(a)$ and we have just shown that $|f^n(0)| \rightarrow \infty$ as $n \rightarrow \infty$, we deduce that there exists a Baker domain $U$ such that $S(a) \subset U$.

We have shown that a single point of the Baker domain $U$, namely $0$, has property \eqref{Iaa}, specifically that it lies in $I_a^a(f)$. To extend this and show all points in the Baker domain are contained in $I_a^a(f)$ we use the following distortion lemma \cite[Lemma 7, equation (5)]{Bergweiler}.

\begin{lem}
    \textit{Let $G$ be an unbounded domain in $\mathbb{C}$ with at least two finite boundary points, and let $f$ be a function analytic in $G$ such that $f(G) \subset G$, $f^n(z) \rightarrow \infty$ as $n \rightarrow \infty$, for $z \in G$, and $f$ does not extend analytically to $\infty$ with $f(\infty) = \infty$. Suppose also that $\hat{\mathbb{C}}\backslash G$ contains a connected set $\Gamma$ such that $\{a, \infty\} \subset \Gamma$ for some $a \in \mathbb{C}$. Then, for any compact subset $K$ of $G$, there exist positive constants $C$ and $n_0$ such that,}

\begin{equation}
    |f^n(z')| \leq C|f^n(z)|, \ \textnormal{for} \ z, z' \in K, \ n \geq n_0.
\end{equation}

\end{lem}

\noindent Hence, by Lemma 3.5 we have that since $\frac{a_n}{5} \leq |f^n(0)| \leq  5 a_n$, there exists for all $z \in U$ positive constants $C(z) > 1$ and $n_0(z) \in \mathbb{N}$ such that,

\begin{equation}
    \nonumber \frac{a_n}{5C(z)} \leq |f^n(z)| \leq 5C(z)a_n, \ \textnormal{for} \ n \geq n_0(z),
\end{equation}

\noindent as required.

Finally, in the proof we assumed that condition \eqref{TheoremCondition} in Theorem 1.1 is true for $n \geq 2$. We can deduce the general case as follows. If \eqref{TheoremCondition} holds for $n \geq n_0$, where $n_0 > 2$, then we put $\tilde{a}_0 = 0$ and
\begin{equation}
    \nonumber \tilde{a}_n = a_{n+n_0-3} +\alpha, \quad \textnormal{for } n \geq 1,
\end{equation}

\noindent where we take $\alpha \in \mathbb{R}$ such that $\tilde{a}_n$ satisfies \eqref{TheoremCondition} for $n \geq 2$. We can then construct a function $f$ with a Baker domain $U$ such that, for $z \in U$, there exists  $C = C(z) > 1$ and $N = N(z) \in \mathbb{N}$ such that
\begin{equation}
    \nonumber \frac{\tilde{a}_n}{C} \leq |f^n(z)| \leq C \tilde{a}_n, \quad \textnormal{for } n \geq N.
\end{equation}

\noindent Now, for $n \geq n_0$,
\begin{equation}
    \nonumber a_n + \alpha \leq \tilde{a}_n = a_{n+n_0-3} + \alpha \leq a_n + (n_0 - 3)(a_n - a_{n-1}) + \alpha \leq n_0 a_n + \alpha,
\end{equation}

\noindent since $a_{n} - a_{n-1} \geq a_{n+1} - a_{n}$ for $n \geq n_0$. Now we can take a constant $D > 1$ such that
\begin{equation}
    \nonumber \frac{a_n}{D} \leq a_n + \alpha, \quad \textnormal{and} \quad n_0a_n + \alpha \leq Dn_0a_n,
\end{equation}

\noindent and so we can conclude that, since $n_0 > 2$,
\begin{equation}
    \nonumber \frac{a_n}{\tilde{C}(z)} \leq |f^n(z)| \leq \tilde{C}(z) a_n, \quad \textnormal{for } n \geq N, 
\end{equation}

\noindent where $\tilde{C}(z) = D n_0 C(z)$.

\section{Proof of remarks}

In this section, we will give two theorems that address Remark 1.2 and Remark~1.3. The first will show that we can apply the logarithm to certain sequences to obtain examples of sequences which satisfy condition \eqref{TheoremCondition}. The second theorem will show that we can find sequences $(a_n)$ tending to $\infty$ arbitrarily slowly and satisfying \eqref{TheoremCondition}. This second theorem allows us to use Theorem 1.1 to construct a function with a Baker domain in which the iterates are controlled from both above and below in a prescribed arbitrarily slow but regular manner.

\begin{thm}
    Let $(a_n)$ be a positive strictly increasing sequence such that $a_n \rightarrow \infty$ as $n \rightarrow \infty$,

    \begin{equation}\label{(1)}
        \frac{a_{n-1} - a_{n-2}}{a_n - a_{n-1}} > 1, \quad \textnormal{for } n \geq 2,
    \end{equation}

    \begin{equation}\label{(2)}
        \frac{a_n}{a_{n-1}} \rightarrow 1 \quad \textnormal{and} \quad  \frac{a_{n-1} - a_{n-2}}{a_n - a_{n-1}} \rightarrow 1 \ \textnormal{ as } n \rightarrow \infty.
    \end{equation}

    \noindent Then the sequence $b_n = \log a_n$ also satisfies \eqref{(1)} and \eqref{(2)}, for all $n \in \mathbb{N}$ such that $\log a_n > 0$.
\end{thm}

\begin{proof}
    Put $\frac{a_n}{a_{n-1}} = 1 + \delta_n$, so $\delta_n > 0$ and $\delta_n \rightarrow 0$ as $n \rightarrow \infty$, and

    \begin{equation}
    \nonumber 
    \begin{split}
        \frac{\delta_{n-1}}{\delta_n} & = \frac{\frac{a_{n-1}}{a_{n-2}} - 1}{\frac{a_n}{a_{n-1}} - 1} = \frac{a_{n-1}}{a_{n-2}} \left( \frac{a_{n-1} - a_{n-2}}{a_n - a_{n-1}} \right) > 1,  \\
        \frac{\delta_{n-1}}{\delta_n} & \rightarrow 1 \ \textnormal{ as } \ n \rightarrow \infty, \quad \textnormal{by \eqref{(1)} and \eqref{(2)}}. 
    \end{split}
    \end{equation}

\noindent Now,

\begin{equation}
    \nonumber \frac{b_{n-1} - b_{n-2}}{b_n - b_{n-1}} = \frac{\log \frac{a_{n-1}}{a_{n-2}}}{\log \frac{a_n}{a_{n-1}}} = \frac{\log(1+ \delta_{n-1})}{\log(1+\delta_n)} > 1,
\end{equation}

\noindent so \eqref{(1)} holds for the $(b_n)$.

To prove part \eqref{(2)} note that

\begin{equation}
    \nonumber \frac{x}{1+x} \leq \log(1+x) \leq x, \quad \textnormal{for } x \geq 0.
\end{equation}

\noindent Hence

\begin{equation}
    \nonumber \frac{\log(1+\delta_{n-1})}{\log(1+\delta_n)} \leq \frac{\delta_{n-1}}{\delta_n/(1+\delta_n)} = \frac{\delta_{n-1}}{\delta_n}(1+\delta_n) \rightarrow 1 \ \textnormal{ as } n \rightarrow \infty
\end{equation}

\noindent and

\begin{equation}
    \nonumber \frac{b_n}{b_{n-1}} = \frac{\log a_n}{\log a_{n-1}} = \frac{\log a_{n-1}(1+\delta_n)}{\log a_{n-1}} = 1 + \frac{\log(1+\delta_n)}{\log a_{n-1}} \rightarrow 1 \ \textnormal{ as } n \rightarrow \infty.
\end{equation}

\noindent So \eqref{(2)} also holds for the $(b_n)$.
\end{proof}

\begin{thm}
    Let $(\alpha_n)$ be a positive non-decreasing sequence such that $\alpha_n \rightarrow \infty$ as $n \rightarrow \infty$. Then there exists a positive increasing sequence $(a_n)$ with $a_n < \alpha_n$ for all sufficiently large $n$, $a_0 = 0$, $a_n \rightarrow \infty$ as $n \rightarrow \infty$ and the $(a_n)$ satisfy

    \begin{equation}\label{b_nCondition}
        1 \leq \frac{a_{n-1} - a_{n-2}}{a_n - a_{n-1}} < 2, \ \textnormal{ for } n \geq 2.
    \end{equation}
\end{thm}

\begin{proof}
    We take $a_1 = \min\{1, \frac{\alpha_1}{3} \}$ and $a_0 = 0$ and define the rest of the $a_n$ as follows,
    \begin{numcases}{a_{n+1} = }
        a_n + (a_n - a_{n-1}), \quad & \textnormal{if} $a_n < \alpha_n - 4$, \label{=1}
        \\
        a_n + \frac{2}{3}(a_n - a_{n-1}), \quad & \textnormal{if} $a_n \geq \alpha_n - 4$. \label{=3/2}
    \end{numcases}

    \noindent From this, we observe that $a_n$ is positive and increasing since $a_1-a_0 = a_1 > 0$. Defining the $(a_n)$ in this way also ensures that they satisfy \eqref{b_nCondition}.

    We need to show that $a_n < \alpha_n$ for all $n \in \mathbb{N}$. First, we note that, if we define the $a_n$ in this way, then $a_{n+1} - a_n \leq a_n - a_{n-1}$. 
    
    Next suppose, for a contradiction, that $a_n \geq \alpha_n - 4$ for all $n \geq N$ say. Then, for $n \geq N$, $a_{n+1} - a_n = \frac{2}{3}(a_n - a_{n-1})$. Hence, for $n > N$,
    \begin{equation}\label{Sums_To_3}
        \begin{split}
            a_n & = (a_n - a_{n-1}) + (a_{n-1} - a_{n-2}) + \cdot\cdot\cdot + (a_{N+1} - a_N) + a_N \\
            & = \sum_{i = N+1}^{n}(a_i - a_{i-1}) + a_N \\
            & = \sum_{i = N+1}^{n} \left(\frac{2}{3}\right)^{i-N}(a_N - a_{N-1}) + a_N \leq 3(a_N - a_{N-1}) + a_N.
        \end{split}
    \end{equation}

    \noindent Since $a_1 - a_0 \leq 1$, by \eqref{=1} and \eqref{=3/2}, $a_N - a_{N-1} \leq 1$, and so, $a_n \leq a_N + 3$ for $n \geq N$. Since $a_n \geq \alpha_n - 4$ for all $n \geq N$ and $\alpha_n \rightarrow \infty$ as $n \rightarrow \infty$ this gives us a contradiction. Hence, there exists $(N_k)$ such that $a_{N_k} < \alpha_{N_k} - 4$ for all $k \in \mathbb{N}$. 
    
    Suppose now that $n \in \mathbb{N}$ is such that $N_k < n < N_{k+1}$. Then by \eqref{Sums_To_3} we have that
    \begin{equation}
        \nonumber a_n < a_{N_k} + 3 < \alpha_{N_k} - 4 + 3 < \alpha_{N_k} \leq \alpha_n,
    \end{equation}

    \noindent since $a_{N_k} - a_{N_k - 1} \leq a_1 - a_0 \leq 1$ and $(\alpha_n)$ is non-decreasing.

    Finally, we need to show that $a_n \rightarrow \infty$ as $n \rightarrow \infty$. We have two cases, either we will apply \eqref{=3/2} infinitely often, or there exists some $N \in \mathbb{N}$ such that for all $n \geq N$, we apply \eqref{=1}. If we apply \eqref{=3/2} infinitely often, then infinitely often we have that $a_n \geq \alpha_n - 4$ and since $\alpha_n \rightarrow \infty$ as $n \rightarrow \infty$, we also have $a_n \rightarrow\infty$ as $n \rightarrow \infty$. If for some $N \in \mathbb{N}$ we apply \eqref{=1} for all $n \geq N$, then we have that $a_{N+m} = a_N + m(a_N - a_{N-1})$ for $m \geq 1$, so $a_n \rightarrow \infty$ as $n \rightarrow \infty$.
\end{proof}

\begin{rem}
    Note that Theorem 4.2 implies that any positive sequence $(a_n)$ that tends to $\infty$ has a concave minorant with respect to $n$ that also tends to $\infty$.
\end{rem}

\newpage

\bibliography{refs}

\end{document}